\documentclass[12pt,reqno]{amsart}
\usepackage{amsmath, amsthm, amssymb}
\usepackage{parskip}

\usepackage[T1]{fontenc}

\numberwithin{equation}{section}

\newtheorem{proposition}{Proposition}[section]
\newtheorem{theorem}[proposition]{Theorem}
\newtheorem{corollary}[proposition]{Corollary}
\newtheorem{lemma}[proposition]{Lemma}

\usepackage{todonotes}

\newcommand\R{\mathbb{R}}

\newcommand\C{\mathbb{C}}

\newcommand{\curl}{\mathrm{curl}\,}

\renewcommand{\div}{\mathrm{div}\,}

\newcommand{\eqnb}{\begin{equation}}
\newcommand{\eqne}{\end{equation}}

\usepackage[colorlinks=true,urlcolor=blue,
citecolor=red,linkcolor=blue,linktocpage,pdfpagelabels,
bookmarksnumbered,bookmarksopen]{hyperref}

\usepackage[left=2cm, right=2cm, top=2cm, bottom=2cm]{geometry}

\usepackage{listings}    % pakiet do wstawiania listingu
\usepackage{xcolor}      % pakiet do kolorowania

\lstdefinestyle{pythonstyle}{
    language=Python,
    basicstyle=\ttfamily\footnotesize,      % wielkość czcionki i krój (typewriter)
    columns=fullflexible,
    keywordstyle=\bfseries\color{blue},      % słowa kluczowe
    commentstyle=\itshape\color{gray},       % komentarze
    stringstyle=\color{olive},               % łańcuchy znaków (stringi)
    showstringspaces=false,                  % nie pokazuj znaku spacji w łańcuchach
    breaklines=true,                         % łamanie długich linii
    numbers=left,                            % numery wierszy z lewej
    numberstyle=\tiny,                       % formatowanie numerów wierszy
    numbersep=5pt,                           % odstęp od numerów wierszy
    tabsize=4,                               % rozmiar tabulacji
    frame=single,                            % ramka wokół kodu
    captionpos=b,                            % podpis pod listingiem
    xleftmargin=1em, xrightmargin=1em        % dodatkowe wcięcia po bokach
}

\begin{document}

\title{Time-zero limits of Kaden's spirals and 2D Euler}

\author[B. Bieganowski]{Bartosz Bieganowski}
\address[B. Bieganowski]{\newline\indent
			Faculty of Mathematics, Informatics and Mechanics, \newline\indent
			University of Warsaw, \newline\indent
			ul. Banacha 2, 02-097 Warsaw, Poland}	
			\email{\href{mailto:bartoszb@mimuw.edu.pl}{bartoszb@mimuw.edu.pl}}	

\author[T. Cieślak]{Tomasz Cieślak}
\address[T. Cieślak]{\newline\indent  	
			Institute of Mathematics,		\newline\indent
			Polish Academy of Sciences, \newline\indent
			ul. \'Sniadeckich 8, 00-956 Warszawa, Poland}
			\email{\href{mailto:cieslak@impan.pl}{cieslak@impan.pl}}

\author[J. Siemianowski]{Jakub Siemianowski}
\address[J. Siemianowski]{\newline\indent  	
			Faculty of Mathematics and Computer Science,		\newline\indent
			Nicolaus Copernicus University, \newline\indent
			ul. Chopina 12/18, 87-100 Toru\'n, Poland}
			\email{\href{mailto:jsiem@mat.umk.pl}{jsiem@mat.umk.pl}}

\begin{abstract}
The present note is devoted to the studies of the relation of the time-zero limits of Kaden's spirals and the 2D Euler equation. It is shown that
the time-zero limits of Kaden's spirals satisfy inhomogeneous 2D Euler in a weak sense. As a corollary, the necessity of both, the decay of spherical averages around the origin of the spiral as well as the velocity matching condition, for the 2D Euler equation to hold in a weak sense, is shown.
Finally, some preliminary results concerning the Kaden spirals are obtained.
\end{abstract}

\maketitle

\section{Introduction}\label{intro}
The present article aims at two goals. On the one hand we shall present a steady-state version of a theorem guaranteeing that a vortex sheet solves 2D Euler in a weak sense in a spirit of a similar theorem stating that self-similar vortex sheets satisfy 2D Euler in a weak sense formulated and proven in \cite{CKO}. The latter theorem was used in \cite{CKO} to show that the celebrated Prandtl spirals introduced in \cite{prandtl} solve 2D Euler in a weak sense. Moreover, we shall take the steady state version as an opportunity to show that two essential conditions emphasized in \cite{CKO}, namely the velocity matching condition and the decay of the spherical averages of the square of the velocity around the origin of the spiral, are necessary conditions. Indeed, we shall give examples of objects satisfying all the assumptions of the theorem on the weak solutions to the 2D Euler, but either the velocity matching or the decay of the spherical averages, and show that such objects violate 2D Euler in a weak sense.

The objects giving rise to examples underlying the necessary role of both, the velocity matching as well as the decays of spherical means of velocities, are interesting in its own way. These are the time-zero limits of the Kaden spirals (see \cite{kaden}). Since Kaden spirals are believed to be objects whose perturbation could lead to the counterexamples of uniqueness in the Delort theorem (see \cite{delort}), see for instance the numerical evidences in \cite{lopesy1}, where the initial data for such perturbations are of importance, it is of particular interest to examine how far are the time-zero limits of Kaden's spirals from being weak steady states of the 2D Euler. The latter is investigated in Theorems \ref{tw_2/3} and \ref{tw_wieksze}. 

Last, but not least, some preliminary results concerning the Kaden spirals are given in the last section.

Kaden's spirals introduced in \cite{kaden} are particular examples of spiral self-similar vortex sheets. For the introduction of vortex sheets we refer for instance to \cite{majda_bertozzi}. Roughly speaking vortex sheets are the divergence-free (in the sense of distributions) velocity fields discontinuous along the curve carrying the vorticity (again in the sense of distributions). Such vector fields are regular off the curve of discontinuity and both, the divergence as well as the vorticity, of such vector fields is zero of the curve. The vorticity of such vector field is then a measure supported at the curve of discontinuity. In the case of Kaden's spirals, the curve of discontinuity is a spiral satisfying the equation
\[
Z(\Gamma, t)=R(\Gamma,t)e^{i\Theta(\Gamma, t)},
\]
where $Z$ is a position in the plane (understood as a complex plane) described by a complex number, while $\Gamma$ is the cumulative vorticity carried by the curve from the origin of the spiral till the considered position. In the case of Kaden's spirals $R$ and $\Gamma$ are given as
\begin{equation}\label{definicje}
R(\Gamma, t)=\Gamma^{\frac{\mu}{2\mu-1}}, \;\; \Theta(\Gamma, t)=\frac{t}{2\pi}\Gamma^{\frac{1}{1-2\mu}},
\end{equation}
where $\mu\in(1/2, 1)$ is a given constant. Moreover, then a density of a vorticity measure $\omega_t$ supported on the spiral $Z$ (which can be parametrized as $s\mapsto s \exp(\frac{it}{2\pi}s^{1-1/\mu})$, compare \eqref{definicje}) is $g_t(s)=(2-1/\mu)s^{1-1/\mu}$. In partcicular, for any $f\in L^1(\R^2, \omega_t)$ we have
\begin{equation}\label{gestosc}
\int_{\R^2} f(x) \, d\omega_t(x)=\int_0^\infty f\left(s\exp\left(\frac{it}{2\pi}s^{-1/\mu}\right)\right)g_t(s)  \, ds,
\end{equation}
see \cite[Section 4]{COPS}.

Next, as we know from \cite[Proposition 4.5]{COPS}, when $t\rightarrow 0^+$ the divergence-free velocity corresponding to $\omega_t(x)$ supported on $Z$ with a density given by \eqref{gestosc} converges in $L^2_{loc}$ to the velocity whose vorticity is (in the sense of distributions)
\begin{equation}\label{stanys}
d\omega_0(x_1+ix_2)=\alpha x_1^{\alpha-1}\chi_{(0,\infty)}(x_1)\delta_0(x_2) \, dx_1 \, dx_2,
\end{equation}
where $\alpha=2-1/\mu$, $\chi$ is a characteristic function. Let us emphasize that the divergence-free velocity field corresponding to $\eqref{stanys}$ is given by the formula
\begin{equation}\label{B-S}
v(x)=\frac{1}{2\pi}\int_{\R^2}\frac{(x-y)^\perp}{|x-y|^2} \, d\omega_0(y),
\end{equation}
where $x=(x_1,x_2)\in \R^2$, $x^\perp=(-x_2,x_1)$. This Biot-Savart formula is non-standard, notice that $\omega_0$ is only a $\sigma$-finite measure, not a finite one, so the Biot-Savart formula requires justification. It is given in \cite[Theorem 2.1 and Proposition 1.4]{COPS}. Since $\omega_0$ satisfies \eqref{stanys}, $\omega_0(B(0,r))=r^{\alpha}$, $0<\alpha<1$. Finally, we notice that the latter implies
\begin{equation}\label{kin_energ}
\int_{B(0,r)}|v|^2 \, dx \leq cr^{2\alpha},
\end{equation}
see \cite[Theorem 1.1]{COPS}. The kinetic energy of the flow is locally finite.

The main results of the present article read as follows.
\begin{theorem}\label{tw_2/3}
Let $\mu=2/3$, moreover assume that $v$ is a divergence-free velocity corresponding to $\omega_0$ given in \eqref{stanys} via the Biot-Savart law \eqref{B-S}. Then $v$ satisfies in the sense of distributions
\begin{equation}\label{Euler_2/3}
\left(v\cdot \nabla\right)v+\nabla p=Y\delta_{(0,0)},
\end{equation}
where $\delta_{(0,0)}$ is the Dirac delta centered at the origin, while $Y=\int_{\partial B(0,1)}\frac{1}{2}|v(x)|^2x-\left(v(x)\cdot x\right)v(x) dS(x)$.
\end{theorem}
For $\mu\in(2/3,1)$ a slightly different inhomogeneous Euler equation is satisfied.
\begin{theorem}\label{tw_wieksze}
Let $\omega_0$ be as in \eqref{stanys} with $\mu\in(2/3,1)$ and assume $v=(v_1,v_2)$ to be a divergence-free velocity given by \eqref{B-S}. Then $v$ satisfies in a weak sense
\begin{equation}\label{Euler_wieksze}
\left(v\cdot \nabla\right)v+\nabla p=(-v_2,0)\gamma \delta_{\Sigma},
\end{equation}
where $\gamma(x_1,x_2)=\alpha x_1^{\alpha-1}\chi_{(0,\infty)}(x_1)\delta_0(x_2)$, $\Sigma$ is a half line $\{(x_1,0) \ : \ x_1\geq 0\}$.
\end{theorem}
The meaning of the above theorems is twofold. On the one hand, as it was already noted, $\omega_0$ is a time-zero limit of Kaden's spirals.
It is expected that some enhancement of Kaden's spirals are the counterexamples to the uniqueness of Delort's solutions. Indeed, numerical simulations suggest that there exist spiral vortex sheets emanating from the flat shear flows. If such shear flows are at the same time Euler's steady states, then the non-uniqueness follows. The above spiral perturbations seem to be some enhancement of Kaden's spirals. For this reason it is of importance to know what sort of equations is satisfied by the time-zero limit of Kaden's spirals. The answer is given in our Theorems \ref{tw_2/3} and \ref{tw_wieksze}.

Next, as a byproduct of Theorems \ref{tw_2/3} and \ref{tw_wieksze}, we obtain that in the case of spiral-like vortex sheets the velocity matching condition and the decay of the spherical averages of the square of the velocity (conditions introduced recently in \cite{CKO}) are necessary conditions for the vortex sheet to be a weak solution to the 2D Euler. Indeed, in \cite{CKO} the self-similar vortex sheets, which are divergence-free as well as curl-free off the support of the vortex sheet, which moreover have locally finite kinetic energy and such that the pressure as well as normal components of velocity are continuous across the support of the vortex sheet, satisfy 2D Euler in a weak sense if
\begin{enumerate}
\item[(i)] $\int_{\partial B(0,r)} \left(|v|^2+p\right)  \, dS(x)\rightarrow 0,\;\mbox{when}\;r\rightarrow 0^+,$
\item[(ii)] $n\cdot (tv-\xi z)_{|\Sigma}=0,$
\end{enumerate}
here $\Sigma$ is a support of the vorticity, $n$ is the normal vector to $\Sigma$ and $\xi>0$ occurs in the self-similarity, i.e. $v(x,t)=t^{\xi-1}w(x/t^{\xi})$ for some profile function $w$.

As we show in Section \ref{trzy} conditions which replace (i) and (ii) in the steady state case, are violated by $\omega_0$ ((i) is violated by $\omega_0$ with $\mu=2/3$, while (ii) is violated by $\omega_0$ with $\mu>2/3$) and the 2D Euler is then not satisfied, despite all the other required assumptions are met. Hence, both (i) and (ii) are natural necessary conditions.

Specially in the light of the last fact one sees the crucial meaning of formula \cite[(1.10)]{COPS}. It allows to compute the spherical averages around the origin of the spiral and thus check the necessary condition (i). The above mentioned formula will be useful in our considerations, it says that
\begin{equation}\label{cops}
\frac{1}{2\pi}\int_0^{2\pi}|v(re^{i\theta})|^2 \, d\theta=\frac{1}{4\pi^2}\left(\sum_{n=0}^\infty r^{-2n-2}|m_{r,n}(\omega)|^2+\sum_{k=1}^\infty r^{2k-2}|M_{r,k}(\omega)|^2\right),
\end{equation}
where $m_{r,n}$ and $M_{r,k}$ are the complex, respectively, inner and outer moments of a $\sigma$-finite measure $\omega$ defined as
\begin{equation}\label{momenty}
m_{r,0}(\omega)=\omega(B(0,r)),\;\;m_{r,n}(\omega)=\int_{B(0,r)}u^n \,  d\omega(u),  \;n\geq 1,
\end{equation}
where $u^n$ is an $n$-th power of a complex number $u$,
\begin{equation}\label{momenty_zewn}
M_{r,k}(\omega)=\int_{\C\setminus B(0,r)}u^{-k} \, d\omega(u),\; k\geq 1,
\end{equation}
again $k$-th powers are taken in a complex sense.

Let us also emphasize the very recent contribution of F. Shao, D. Wei and Z. Zhang \cite{shao} concerning algebraic spiral vortex sheets as solutions to the 2D Euler.

\section{Main lemma}\label{main_l}
The present section is devoted to the formulation and a proof of a lemma concerning the weak solvability of 2D Euler equation by the steady state vector fields. For simplicity we consider only a vortex sheet of the form $\Sigma=\{(x_1,0) \ : \ x_1\geq 0\}$.

We say that $v=(v_1,v_2)$ satisfies the 2D steady Euler in a weak sense if for any divergence-free $C_0^\infty$ test function $\phi:\R^2\rightarrow \R^2$ there holds
\[
\int_{\R^2}\sum_{i,j=1}^2 v_iv_j\partial_i\phi_j \,  dx=0,
\]
and for any $C_0^\infty$ test function $\psi:\R^2\rightarrow \R$
\[
\int_{\R^2}v\cdot \nabla \psi \,  dx=0.
\]
Let us consider a steady vector field $v$ with the following properties:
\begin{eqnarray}
&&v \;\mbox{is smooth off the line $\Sigma$ (and continuous from both sides up to $\Sigma$)},
\nonumber\\
&&\curl v=\div v=0\;\mbox{outside}\;\Sigma,\label{raz}\\
&&(v^+-v^-)\cdot n_{|\Sigma}=0,\label{normal}\\
&&\int_{B(0,r)}|v|^2 \, dx\rightarrow 0\;\mbox{if}\;r\rightarrow 0,\label{dwa}
\end{eqnarray}
where $n=(0,1)$ is a normal vector to $\Sigma$, $v^+, v^-$ are the limits of $v$ at opposite sides of $\Sigma$.
\begin{lemma}\label{lemat}
The vector field $v$ satisfying \eqref{raz}-\eqref{dwa} fulfils
\begin{eqnarray}\label{zlematu}
&&\int_{\R^2}\sum_{i,j=1}^2 v_iv_j\partial_i\phi_j  \, dx=\lim_{r\rightarrow 0}\left(\int_{\partial B(0,r)}(v\cdot n)(v\cdot \phi) \, dS(x)
-\int_{\partial B(0,r)}\frac{1}{2}|v|^2n\cdot \phi\, dS(x)\right) \nonumber\\
&+&\lim_{r\rightarrow 0}\left(\int_{(r,\infty)}\frac{1}{2}\left(|v^+|^2-|v^-|^2\right)\phi_1 \, dx_1
+\int_{(r,\infty)}\phi \cdot (v^+ - v^-)v_2^+ \, dx_1 \right).
\end{eqnarray}
\end{lemma}
\begin{proof}
By \eqref{dwa} we have
\begin{eqnarray*}
&&\int_{\R^2}\sum_{i,j=1}^2 v_iv_j\partial_i\phi_j \,  dx=\lim_{r\rightarrow 0}\int_{B(0,r)^c}\sum_{i,j=1}^2 v_iv_j\partial_i\phi_j \,  dx\\
&=&\lim_{r\rightarrow 0}\left(-\int_{B(0,r)^c}\sum_{i,j=1}^2v_i\partial_iv_j\phi_j \, dx+\int_{(r,\infty)}\sum_{i,j=1}^2(v_i^+v_j^+-v_i^-v_j^-)n_i\phi_j \, dx_1 +\int_{\partial B(0,r)}(v\cdot n)(v\cdot \phi) \, dS(x)\right)\\
&=&\lim_{r\rightarrow 0}\left(\int_{(r,\infty)}\frac{1}{2}\left(|v^+|^2-|v^-|^2\right) \phi_1 \, dx_1 -\int_{\partial B(0,r)}\frac{1}{2}|v|^2n\cdot \phi  \, dS(x) \right)\\
&+&\lim_{r\rightarrow 0}\left(\int_{(r,\infty)}\phi \cdot (v^+ - v^-)v_2^+ \, dx_1 +\int_{\partial B(0,r)}(v\cdot n)(v\cdot \phi) \, dS(x) \right),
\end{eqnarray*}
where we used \eqref{normal} and that \eqref{raz} yields $(v\cdot\nabla)v=\nabla \left(\frac{1}{2}|v|^2\right)$ outside $\Sigma$. The lemma is proven.
\end{proof}
Using Lemma \ref{lemat}, we immediately observe the following.
\begin{corollary}\label{Euler}
Assume that $v$ satisfies \eqref{raz}-\eqref{dwa}, moreover assume
\begin{equation}\label{cisnienie}
v_1^+(x_1,0)=v_1^-(x_1,0)\;\mbox{for any}\;\;x_1>0.
\end{equation}
Next assume
\begin{eqnarray}
&&\int_{\partial B(0,r)}|v|^2 \,  dS(x)\rightarrow 0\;\mbox{when}\;r\rightarrow 0,\label{decay}\\
&&v_2^+(x_1, 0)=0\;\mbox{for}\;\;x_1>0.\label{vel_matching}
\end{eqnarray}
Then $v$ satisfies steady 2D Euler in a weak sense.
\end{corollary}
Condition \eqref{decay} is a decay of spherical averages of squares of velocities around the origin corresponding to (i) in Section \ref{intro}. Condition \eqref{vel_matching} corresponds to the velocity matching condition in \cite{CKO}, compare to (ii) in Section \ref{intro}. We also notice that \eqref{cisnienie} is nothing else but the lack of jump of the pressure across the vortex sheet.

Finally, notice that in the case of steady flows, the velocity matching condition has a very clear physical interpretation. Namely, for the vortex sheet to stay steady, vorticity support curve should keep its position. This is exactly \eqref{vel_matching}.

\section{Zero-time limits of Kaden's spirals}\label{trzy}
In this section we use Lemma \ref{lemat} to obtain Theorems \ref{tw_2/3} and \ref{tw_wieksze}. Before we turn to the proofs, we state and prove some
essential observations. We have the following fact.
\begin{proposition}\label{wlasnosci}
Assume $\omega_0$ is a $\sigma$-finite measure given by \eqref{stanys}. Next, consider $v$ corresponding to $\omega_0$ via \eqref{B-S}.
By the results in \cite{COPS} we know that $v$ is a $L^2_{loc}$ function, divergence-free in the sense of distributions. Moreover,  in the sense of distributions $\curl v=\omega_0$,  $v$ is regular off the support of $\omega_0$ and $\div v=\curl v=0$ for $x\in \R^2\setminus \{(x_1,0):x_1\geq 0\}$.

Next, $v$ satisfies the following properties:
\begin{equation}\label{decays}
\int_{\partial B(0,r)}|v|^2 \, dS(x)=\int_0^{2\pi}r|v(re^{i\theta})|^2 \, d\theta=\frac{\pi\alpha^2}{2\sin^2(\pi\alpha)}r^{2\alpha-1},
\end{equation}
\begin{eqnarray}\label{naprostej}
v^+(s,0)&=&\frac{\alpha}{2}\left(-s^{\alpha-1}, \frac{1}{\pi}PV\int_0^\infty\frac{t^{\alpha-1}}{s-t} \, dt\right),\\
v^-(s,0)&=&\frac{\alpha}{2}\left(s^{\alpha-1}, \frac{1}{\pi}PV\int_0^\infty\frac{t^{\alpha-1}}{s-t} \, dt\right), \;\mbox{for}\;s>0.\nonumber
\end{eqnarray}
\end{proposition}
\noindent
\begin{proof}[Proof of Proposition \ref{wlasnosci}]
First assertions follow \cite[Theorem 2.1 and Proposition 1.4]{COPS}. The claim \eqref{decays} is computed
in \cite[Section 3.1]{COPS}. It is also implied by the following fact: for any $t>0$ and any $x\in\R^2$ we have
\begin{equation}\label{homo}
v(tx)=t^{\alpha-1}v(x).
\end{equation}
Indeed,
\begin{eqnarray*}
v(tx)&=&\int_{\R^2}\frac{(tx-y)^\perp}{|tx-y|^2}d\omega_0(y)=\int_0^\infty \frac{(tx_1-s, tx_2)^\perp}{(tx_1-s)^2+(tx_2)^2}\alpha s^{\alpha-1}ds\\
&=&\int_0^\infty\frac{(tx_1-t\xi, tx_2)^\perp}{(tx_1-t\xi)^2+(tx_2)^2}\alpha (t\xi)^{\alpha-1}td\xi=t^{\alpha-1}\int_0^\infty\frac{(x_1-\xi,x_2)^\perp}{(x_1-\xi)^2+x_2^2}\alpha \xi^{\alpha-1}d\xi\\
&=&t^{\alpha-1}v(x).
\end{eqnarray*}
The only thing that is left to justify in Proposition \ref{wlasnosci} is \eqref{naprostej}. Formula \eqref{naprostej} reduces to proving the following two identities:
\begin{equation}\label{harm}
\lim_{x_2\rightarrow 0^\pm}\frac{1}{2\pi}\int_0^\infty\frac{x_2s^{\alpha-1}}{(x_1-s)^2+x_2^2}ds=\pm \frac{1}{2} x_1^{\alpha-1}
\end{equation}
and
\begin{equation}\label{anharm}
\lim_{x_2\rightarrow 0}\frac{1}{2\pi}\int_0^\infty \frac{(x_1-s)s^{\alpha-1}}{(x_1-s)^2+x_2^2}ds=\frac{1}{2\pi}PV\int_0^\infty\frac{s^{\alpha-1}}{x_1-s}ds
\end{equation}
for $x_1>0$, $\alpha\in(0,1)$.

We first show \eqref{harm}. We fix $x_1>0$, next we fix
$0<\epsilon<\frac{x_1}{2}.$
Without loss of generality we assume $x_2>0$. We calculate
\begin{align*}
\frac{1}{2\pi}&\int_0^\infty \frac{x_2s^{\alpha-1}}{(x_1-s)^2+x_2^2}ds=\frac{1}{2\pi}\left(\int_0^\epsilon +\int_\epsilon^\infty\right)\\
\leq\frac{1}{2\pi}&\frac{x_2}{(x_1/2)^2}\int_0^\epsilon s^{\alpha-1}ds+\frac{1}{2\pi}\int_{\epsilon}^\infty \frac{x_2s^{\alpha-1}}{(x_1-s)^2+x_2^2}ds\\
\leq \frac{1}{2\pi}& \frac{1}{\alpha}\epsilon^{\alpha}\frac{x_2}{(x_1/2)^2}+\frac{1}{2\pi}\int_{\epsilon}^\infty \frac{s^{\alpha-1}}{x_2}\frac{1}{\left(\frac{x_1-s}{x_2}\right)^2+1}ds.
\end{align*}
The first term on the right-hand side goes to $0$ with $x_2\rightarrow 0$ independently of $\epsilon$. Let us focus on the second term.
\[
\int_{\epsilon}^\infty \frac{s^{\alpha-1}}{x_2}\frac{1}{\left(\frac{x_1-s}{x_2}\right)^2+1}ds=\int_{\frac{x_1-\epsilon}{x_2}}^{-\infty}\frac{(x_1-\xi x_2)^{\alpha-1}}{\xi^2+1}d\xi
\]
\[
=-\int_{-\infty}^{\frac{x_1-\epsilon}{x_2}}\frac{(x_1-\xi x_2)^{\alpha-1}}{\xi^2+1}d\xi.
\]
By the Lebesgue monotone convergence theorem we infer that (notice that $(x_1-\xi x_2)^{\alpha-1}\leq \epsilon^{\alpha-1}$) with $x_2\rightarrow 0$ the last integral tends to $x_1^{\alpha-1} \left(\arctan(\infty)-\arctan(-\infty) \right)=x_1^{\alpha-1}\pi$. So that \eqref{harm} holds.

To justify \eqref{anharm} we calculate
\[
\int_0^\infty \frac{(x_1-s)s^{\alpha-1}}{(x_1-s)^2+x_2^2}ds=\int_0^{x_1-\epsilon}+\int_{x_1-\epsilon}^{x_1+\epsilon}
+\int_{x_1+\epsilon}^\infty.
\]
The limit $x_2\rightarrow 0$ of the first and third terms, followed by the limit $\epsilon\rightarrow 0$ gives the desired result. We only need to justify that the limit
\[
\lim_{\epsilon\rightarrow 0}\lim_{x_2\rightarrow 0}\int_{x_1-\epsilon}^{x_1+\epsilon}\frac{(x_1-s)s^{\alpha-1}}{(x_1-s)^2+x_2^2}\,ds=0.
\]
Indeed, note that
$$
\int_{x_1-\epsilon}^{x_1+\epsilon}\frac{(x_1-s)s^{\alpha-1}}{(x_1-s)^2+x_2^2}\,ds = \int_{x_1-\epsilon}^{x_1}\frac{(x_1-s)s^{\alpha-1}}{(x_1-s)^2+x_2^2}\,ds - \int_{x_1-\epsilon}^{x_1}\frac{(x_1-s)(2x_1-s)^{\alpha-1}}{(x_1-s)^2+x_2^2}\,ds,
$$
where in the second integral we did the obvious change of variables $s \mapsto 2x_1 - s$. Hence
$$
\int_{x_1-\epsilon}^{x_1+\epsilon}\frac{(x_1-s)s^{\alpha-1}}{(x_1-s)^2+x_2^2}\,ds = \int_{x_1-\epsilon}^{x_1}\frac{(x_1-s)\left(s^{\alpha-1}-(2x_1-s)^{\alpha-1}\right)}{(x_1-s)^2+x_2^2}\,ds.
$$
Thus, it is sufficient to check that the function $s \mapsto \frac{(x_1-s)\left(s^{\alpha-1}-(2x_1-s)^{\alpha-1}\right)}{(x_1-s)^2} = \frac{s^{\alpha-1}-(2x_1-s)^{\alpha-1}}{x_1-s}$ is integrable in the left neighbourhood of $x_1$ and then the statement follows easily from the dominated convergence theorem. Indeed, from the mean value theorem (note that we have $s \leq 2x_1 -s$)
$$
\int_{x_1-\epsilon}^{x_1} \left| \frac{s^{\alpha-1}-(2x_1-s)^{\alpha-1}}{x_1-s} \right| \,ds =  \int_{x_1-\epsilon}^{x_1} \left| \frac{(\alpha - 1) \theta_s^{\alpha-2} (2x_1-2s)}{x_1-s} \right| \, ds = (2-2\alpha)  \int_{x_1-\epsilon}^{x_1} \theta_s^{\alpha-2} \,ds,
$$
where $\theta_s \in [s, 2x_1-s]$. In particular
\begin{equation}
\label{eps}
\int_{x_1-\epsilon}^{x_1} \left| \frac{s^{\alpha-1}-(2x_1-s)^{\alpha-1}}{x_1-s} \right| \,ds  \lesssim \int_{x_1-\epsilon}^{x_1} s^{\alpha-2} \, ds = \frac{x_1^{\alpha-1}- (x_1-\epsilon)^{\alpha-1}}{\alpha-1}.
\end{equation}
Thus $s \mapsto \frac{s^{\alpha-1}-(2x_1-s)^{\alpha-1}}{x_1-s}$ is integrable on $[x_1-\epsilon, x_1]$ and we have the following inequality
$$
\left| \frac{(x_1-s)\left(s^{\alpha-1}-(2x_1-s)^{\alpha-1}\right)}{(x_1-s)^2+x_2^2}\,ds \right| \leq \frac{s^{\alpha-1}-(2x_1-s)^{\alpha-1}}{x_1-s}, \quad s \in [x_1-\epsilon, x_1].
$$
From dominated convergence theorem we get that
\begin{align*}
\lim_{x_2 \to 0} \int_{x_1-\epsilon}^{x_1+\epsilon}\frac{(x_1-s)s^{\alpha-1}}{(x_1-s)^2+x_2^2}\,ds &= \lim_{x_2 \to 0} \int_{x_1-\epsilon}^{x_1}\frac{(x_1-s)\left(s^{\alpha-1}-(2x_1-s)^{\alpha-1}\right)}{(x_1-s)^2+x_2^2}\,ds \\
&= \int_{x_1-\epsilon}^{x_1} \frac{s^{\alpha-1}-(2x_1-s)^{\alpha-1}}{x_1-s} \,ds.
\end{align*}
Moreover, from \eqref{eps},
$$
\lim_{\epsilon \to 0^+} \lim_{x_2 \to 0}\int_{x_1-\epsilon}^{x_1+\epsilon}\frac{(x_1-s)s^{\alpha-1}}{(x_1-s)^2+x_2^2}\,ds = \lim_{\epsilon\to 0^+} \int_{x_1-\epsilon}^{x_1} \frac{s^{\alpha-1}-(2x_1-s)^{\alpha-1}}{x_1-s} \,ds = 0.
$$
\end{proof}

Note that \eqref{decays} immediately implies \eqref{decay} for $\alpha > \frac12$. Similarly we see from \eqref{decays} that the integral $\int_{\partial B(0,r)} |v|^2 \, dS(x)$ is constant (doesn't depend on $r$) for $\alpha = \frac12$ and \eqref{decay} cannot hold. Notice that \eqref{decays} follows twofold. On the one hand it is a consequence of the general formula \eqref{cops} obtained in \cite{COPS}. On the other hand \eqref{decays} is a consequence of \eqref{homo}, a particular property of $v$ given by \eqref{B-S} with $\omega_0$ as in \eqref{stanys}. The more general approach will be utilized in Proposition \ref{lem:srednieZbiezne}.

\begin{lemma}\label{lem:pomoc}
Let $\alpha \in (0,1)$. Then
$$
PV \int_0^\infty \frac{t^{\alpha-1}}{1-t} \, dt = 0
$$
if and only if $\alpha = \frac12$.
\end{lemma}

\begin{proof}
Let us fix $\epsilon>0$ and consider
\[
\int_{1+\epsilon}^\infty \frac{t^{\alpha-1}}{1-t}\, dt = \left [ \begin{aligned}
t &= \frac{1}{\xi}\\
dt &= -\frac{1}{\xi^2}d\xi
\end{aligned}\right]
\int_0^{\frac{1}{1+\epsilon}}\frac{\xi^{1-\alpha}}{1 - \xi^{-1}}\frac{1}{\xi^2}\, d\xi = -\int_0^{\frac{1}{1+\epsilon}}\frac{\xi^{-\alpha}}{1-\xi}\, d\xi.
\]
By the above and since $1-\epsilon < \frac{1}{1+\epsilon}$ we have
\[
\int_0^{1-\epsilon}\frac{t^{\alpha-1}}{1-t}\, dt +  \int_{1+e}^\infty \frac{t^{\alpha-1}}{1-t}\, dt =
\int_0^{1-\epsilon}\frac{t^{\alpha-1}-t^{-\alpha}}{1-t}\, dt - \int_{1-\epsilon}^\frac{1}{1+\epsilon}\frac{t^{-\alpha}}{1-t}\, dt.
\]
Let us show that the last integral tends to $0$ as $\epsilon\to 0^+$:
\[
\begin{aligned}
\left | \int_{1-\epsilon}^\frac{1}{1+\epsilon}\frac{t^{-\alpha}}{1-t}\, dt\right | & \leq (1-\epsilon)^{-\alpha}\int_{1-\epsilon}^\frac{1}{1+\epsilon}\frac{1}{1-t}\, dt  =(1-\epsilon)^{-\alpha}\left( -\log\left ( 1 - \frac{1}{1+\epsilon}\right )
 + \log(\epsilon) \right)\\
 &= (1-\epsilon)^{-\alpha}\log\left( 1 + \epsilon\right)\to 0\quad \text{as }\epsilon\to 0^+.
 \end{aligned}
\]
Therefore
\[
PV \int_0^\infty\frac{t^{\alpha-1}}{1-t}\, dt = \lim_{\epsilon\to 0^+}\int_0^{1-\epsilon}\frac{t^{\alpha-1}-t^{-\alpha}}{1-t}\, dt.
\]
The above integral is equal to $0$ if and only if $\alpha = \frac{1}{2}$. Otherwise we have for any $t\in (0,1)$
\[
t^{\alpha-1}- t^{-\alpha} = t^{-\alpha}\left( t^{2\alpha-1} - 1\right) \begin{cases} < 0&\text{for } \alpha\in \left ( \frac{1}{2},1\right),\\
>0 &\text{for } \alpha\in \left ( 0,\frac{1}{2}\right ).
\end{cases}
\]
We can conclude that
\[
PV \int_0^\infty \frac{t^{\alpha-1}}{1-t} \, dt \begin{cases}< 0&\text{for } \alpha\in \left ( \frac{1}{2},1\right),\\
=0&\text{for }\alpha = \frac{1}{2},\\
>0 &\text{for } \alpha\in \left ( 0,\frac{1}{2}\right )
\end{cases}
\]
\end{proof}

\begin{proof}[Proof of Theorem \ref{tw_2/3}]
We have that $\mu = \frac{2}{3}$, so $\alpha = \frac{1}{2}$. Note that, from Lemma \ref{lemat} we get
\begin{eqnarray*}
&&\int_{\R^2}\sum_{i,j=1}^2 v_iv_j\partial_i\phi_j  \, dx=\lim_{r\rightarrow 0}\left(\int_{\partial B(0,r)}(v\cdot n)(v\cdot \phi) \, dS(x)
-\int_{\partial B(0,r)}\frac{1}{2}|v|^2n\cdot \phi\, dS(x)\right) \\
&+&\lim_{r\rightarrow 0}\left(\int_{(r,\infty)}\frac{1}{2}\left(|v^+|^2-|v^-|^2\right) \phi_1 \, dx_1
+\int_{(r,\infty)} \phi \cdot (v^+ - v^-) v_2^+ \, dx_1 \right).
\end{eqnarray*}
Note that \eqref{naprostej} implies that $|v^+|=|v^-|$ on $(r, \infty)$ for any $r > 0$, and therefore
$$
\int_{(r,\infty)}\frac{1}{2}\left(|v^+|^2-|v^-|^2\right) \phi_1 \, dx_1  = 0.
$$
Note that, since $v_2 := v_2^+=v_2^-$, using \eqref{naprostej} and Lemma \ref{lem:pomoc}, we get
$$
\int_{(r,\infty)} \phi \cdot (v^+ - v^-) v_2^+ \, dx_1 = \int_{(r,\infty)}(v_1^+-v_1^-)v_2^+\phi_1 \, dx_1 = \int_r^\infty \alpha s^{\alpha-1} v_2(s,0) \phi_1(s,0) \, ds = 0,
$$
since \eqref{homo} implies
$$
v_2(s,0) = s^{\alpha-1} v_2(1,0) = s^{\alpha-1} PV \int_0^\infty \frac{t^{\alpha-1}}{1-t} \, dt  =0.
$$
Thus
$$
\int_{\R^2}\sum_{i,j=1}^2 v_iv_j\partial_i\phi_j  \, dx=\lim_{r\rightarrow 0}\left(\int_{\partial B(0,r)}(v\cdot n)(v\cdot \phi) \, dS(x)-\int_{\partial B(0,r)}\frac{1}{2}|v|^2n\cdot \phi\, dS(x)\right).
$$
Note that here $n = \frac{x}{r}$. Moreover we use \eqref{homo} and the change of variables $x \mapsto \frac{x}{r}$ and noting that $n$ is outer normal to the complement of a ball, to see that
\begin{align*}
&\quad \lim_{r\rightarrow 0}\left(\int_{\partial B(0,r)}(v\cdot n)(v\cdot \phi) \, dS(x)-\int_{\partial B(0,r)}\frac{1}{2}|v|^2n\cdot \phi\, dS(x)\right)  \\
&= \lim_{r\rightarrow 0}\left(-\int_{\partial B(0,1)}(v(x)\cdot x)(v (x)\cdot \phi(rx)) \, dS(x)+\int_{\partial B(0,1)}\frac{1}{2}|v(x)|^2x\cdot \phi(rx)\, dS(x)\right),
\end{align*}
which gives the desired result.
\end{proof}

\begin{proof}[Proof of Theorem \ref{tw_wieksze}]
From Lemma \ref{lemat} we have
\begin{eqnarray*}
&&\int_{\R^2}\sum_{i,j=1}^2 v_iv_j\partial_i\phi_j  \, dx=\lim_{r\rightarrow 0}\left(\int_{\partial B(0,r)}(v\cdot n)(v\cdot \phi) \, dS(x)
-\int_{\partial B(0,r)}\frac{1}{2}|v|^2n\cdot \phi\, dS(x)\right) \\
&+&\lim_{r\rightarrow 0}\left(\int_{(r,\infty)}\frac{1}{2}\left(|v^+|^2-|v^-|^2\right) \phi_1 \, dx_1
+\int_{(r,\infty)} \phi \cdot (v^+ - v^-) v_2^+  \, dx_1 \right).
\end{eqnarray*}
Note that, since $\alpha > \frac12$, we have \eqref{decay}. This immediately implies that
$$
\lim_{r\rightarrow 0}\left(\int_{\partial B(0,r)}(v\cdot n)(v\cdot \phi) \, dS(x)
-\int_{\partial B(0,r)}\frac{1}{2}|v|^2n\cdot \phi\, dS(x)\right) = 0.
$$
Moreover, since $|v^+| = |v^-|$, we also have
$$
\int_{(r,\infty)}\frac{1}{2}\left(|v^+|^2-|v^-|^2\right) \phi_1 \, dx_1 = 0.
$$
Since $v_2 := v_2^+ = v_2^-$ we also have
$$
\int_{(r,\infty)} \phi \cdot (v^+ - v^-) v_2^+  \, dx_1 = \int_{(r,\infty)}(v_1^+-v_1^-)v_2^+\phi_1 \, dx_1 = \int_r^\infty \alpha s^{\alpha-1} v_2(s,0) \phi_1(s,0) \, ds,
$$
however Lemma \ref{lem:pomoc} implies that $v_2(s,0) \neq 0$. Hence, the limit
$$
\lim_{r \to 0} \int_r^\infty \alpha s^{\alpha-1} v_2(s,0) \phi_1(s,0) \, ds
$$
may be considered as a distribution $(-v_2,0) \gamma \delta_\Sigma$ acting on a divergence-free $\phi$, where $\gamma$ is a measure on $\Sigma$ given by
$$
\gamma = \alpha x_1^{\alpha-1} \chi_{(0,\infty)} (x_1) \delta_0 (x_2).
$$
\end{proof}

As already noted in the introduction, Theorems \ref{tw_2/3} and \ref{tw_wieksze} indicate that, respectively, the decay of spherical averages over the origin of the spiral as well as velocity matching condition are necessary for the velocity field, corresponding to a vortex sheet, to satisfy the 2D Euler.

Observe that the Dirac delta on the right hand side of \eqref{Euler_2/3} in Theorem \ref{tw_2/3} can be interpreted as an impuls acting on the fluid, its intensity being given by $Y=\int_{\partial B(0,1)}\frac{1}{2}|v(x)|^2x-\left(v(x)\cdot x\right)v(x) dS(x)$. We remark that it is not the first result translating the impulse on the fluid acting on the edge of a sheet into the Dirac delta appearing on the right hand side of the 2D Euler, see \cite{lopesy}. We refer to \cite{lopesy} for the physical interpretation of the Prandtl-Munk vortex sheet \cite{munk} based on considerations of Saffman \cite{safman}.

\section{Kaden's spirals partial results}

This section is devoted to the analysis of the velocity field corresponding to the vorticity given by the self-similar Kaden's spiral, see \cite{kaden}. Based on the results concerning the time-zero limits of the Kaden's spirals, see Theorem \ref{tw_2/3}, we are tempted to state a conjecture concerning the inhomogeneous Euler equation satisfied by the Kaden's spiral in the case $\mu=\frac23$. As we show in what follows, the spherical averages over the spiral origin do not converge to $0$ with the radius $r\to0^+$. However, we're only able to reduce the conjecture to the integral, which appears below. We need an exact value of the integral which seems to be highly oscillatory and apparently it cannot be treated in a straightforward manner via the residue theorem.

Let $\Sigma_t$ denote the Kaden's spiral whose position is described by \eqref{definicje}, and the density of the vorticity at the spiral given by \eqref{gestosc}. Let $v$ be the velocity field corresponding to $\Sigma_t$ via the Biot-Savart law \eqref{B-S}, see \cite[Theorem 2.1 and Proposition 1.4]{COPS}. We can easily compute that $v(\cdot, t) \in C^1 (\R^2 \setminus \Sigma_t)$ for any $t > 0$. Therefore, from \cite[Theorem 2.1]{COPS} it follows that $\div v = 0$ and $\curl v = \omega_t$ in the sense of distributions, and $\div v = 0$, $\curl v = 0$ on $\R^2 \setminus \Sigma_t$. Moreover a simple change of variables gives that
\begin{equation}\label{eq:scaling}
v(x, t)= t^{\mu-1} v \left(\frac{x}{t^\mu}, 1\right).
\end{equation}
Indeed, put
$$
w(x) := v(x,1) = \left( 2 - \frac{1}{\mu} \right) \frac{i}{2\pi} \int_0^\infty \frac{s^{1-1/\mu}}{\overline{x} - s \exp \left( - \frac{i}{2\pi} s^{-1/\mu} \right)} \, ds.
$$
Then, to get \eqref{eq:scaling} we compute
\begin{align*}
t^{\mu-1} w \left( \frac{x}{t^\mu} \right) &= t^{\mu-1} \left( 2 - \frac{1}{\mu} \right) \frac{i}{2\pi} \int_0^\infty \frac{s^{1-1/\mu}}{\overline{x} t^{-\mu} - s \exp \left( - \frac{i}{2\pi} s^{-1/\mu} \right)} \, ds \\
&= t^{\mu-1} \left( 2 - \frac{1}{\mu} \right) \frac{i}{2\pi} \int_0^\infty \frac{s^{1-1/\mu} }{\overline{x}  - st^{\mu} \exp \left( - \frac{i}{2\pi} s^{-1/\mu} \right)} \, t^{\mu}ds \\
%&= [ \zeta := s t^\mu, \ d\zeta = t^\mu  ds ] \\
&= t^{\mu-1} \left( 2 - \frac{1}{\mu} \right) \frac{i}{2\pi} \int_0^\infty \frac{\zeta^{1-1/\mu} t^{-\mu+1} }{\overline{x}  - \zeta \exp \left( - \frac{it}{2\pi} \zeta^{-1/\mu} \right)} \, d\zeta = v(x,t).
\end{align*}
Therefore $v$ is completely defined by the profile $w$.

\begin{lemma}\label{lem:est}
Let $\mu \in (0,1)$ and $n \geq 1$. The following properites hold true
\begin{align}
\label{est1} &|m_{r,n}(\omega_t)| \lesssim \frac{1}{n+2-\frac{1}{\mu}} r^{n+2-\frac{1}{\mu}}, \\
\label{est2} &|m_{r,0}(\omega_t)| = r^{2-\frac{1}{\mu}}, \\
\label{est3} &|M_{r,n}(\omega_t)| \lesssim -\frac{1}{-n+2-\frac{1}{\mu}} r^{-n+2-\frac{1}{\mu}}.
\end{align}
\end{lemma}

\begin{proof}
Using \eqref{gestosc}, we get
\begin{align*}
|m_{r,n}(\omega_t)| &= \left| \int_{B(0,r)} u^n \, d\omega_t (u) \right| = \left(2 - \frac{1}{\mu}\right) \left| \int_0^r s^{n+1-\frac{1}{\mu}} \exp \left( \frac{itn}{2\pi} s^{-1/\mu} \right) \, ds \right| \\
&\lesssim \int_0^r s^{n+1-\frac{1}{\mu}} \, ds = \frac{1}{n+2-\frac{1}{\mu}} r^{n+2-\frac{1}{\mu}},
\end{align*}
which proves \eqref{est1}. \eqref{est2} follows directly from \eqref{momenty}. To show \eqref{est3} we compute similarly
\begin{align*}
|M_{r,n}(\omega_t)| &= \left| \int_{\C \setminus B(0,r)} u^{-n} \, d\omega_t (u) \right| = \left(2 - \frac{1}{\mu}\right) \left| \int_r^\infty s^{-n+1-\frac{1}{\mu}} \exp \left( -\frac{itn}{2\pi} s^{-1/\mu} \right) \, ds \right| \\
&\lesssim \int_r^\infty s^{-n+1-\frac{1}{\mu}} \, ds = - \frac{1}{-n+2-\frac{1}{\mu}} r^{-n+2-\frac{1}{\mu}}.
\end{align*}
\end{proof}

\begin{proposition}\label{lem:srednieZbiezne}
Let $\mu \in \left( \frac23, 1 \right)$. Let $v$ be the velocity field related to the measure $\omega_t$ by the Biot-Savart law. Then
\begin{equation}\label{srednieZbiezne}
\lim_{r \to 0^+} \int_{\partial B(0,r)} |v|^2 \, dS(x) = 0.
\end{equation}
\end{proposition}

\begin{proof}
Observe that
\begin{align*}
\int_{\partial B(0,r)} |v|^2 \, dS(x) &=  r \int_0^{2\pi} \left| v( r e^{i \theta} ) \right|^2 \, d\theta.
\end{align*}
We recall that, from \cite[Theorem 1.2]{COPS}, we have
\begin{align*}
\int_0^{2\pi} \left| v( re^{i\theta} ) \right|^2 \, d\theta  = \frac{1}{2\pi} \left( \sum_{n=0}^\infty r^{-2n-2} |m_{r,n}(\omega_t)|^2 + \sum_{n=1}^\infty r^{2n-2} |M_{r,n} (\omega)|^2 \right)
\end{align*}
and from Lemma \ref{lem:est} we obtain the following estimate
\begin{align*}
\int_0^{2\pi} \left| v( re^{i\theta} ) \right|^2 \, d\theta &\lesssim \sum_{n=1}^\infty r^{-2n-2} |m_{r,n}(\omega_t)|^2 + r^{-2} |m_{r,0}(\omega_t)|^2 + \sum_{n=1}^\infty r^{2n-2} |M_{r,n} (\omega)|^2 \\
&\lesssim \sum_{n=1}^\infty \frac{1}{\left(n+2-\frac{1}{\mu}\right)^2} r^{2-\frac{2}{\mu}} + r^{2-\frac{2}{\mu}} + \sum_{n=1}^\infty \frac{1}{\left(-n+2-\frac{1}{\mu}\right)^2} r^{2-\frac{2}{\mu}} \\
&\lesssim r^{2-\frac{2}{\mu}}.
\end{align*}
Finally
$$
r \int_0^{2\pi} \left| v( re^{i\theta} ) \right|^2 \, d\theta \lesssim r^{3-\frac{2}{\mu}} \to 0 \mbox{ as } r \to 0^+, \quad \mbox{since } \mu \in \left( \frac23, 1 \right).
$$
\end{proof}

We observe that \eqref{srednieZbiezne} doesn't hold for $\mu = \frac{2}{3}$. Indeed, using \cite[Theorem 1.2]{COPS} we get
\begin{align*}
\int_{\partial B(0,r)} |v|^2 \, dS(x) &=  r \int_0^{2\pi} \left| v( r e^{i \theta} ) \right|^2 \, d\theta = \frac{r}{2\pi} \left( \sum_{n=0}^\infty r^{-2n-2} |m_{r,n}(\omega_t)|^2 + \sum_{n=1}^\infty r^{2n-2} |M_{r,n} (\omega)|^2 \right) \\
&\gtrsim r^{-1} | m_{r,0} (\omega_t) |^2 = r^{3-\frac{2}{\mu}} = 1.
\end{align*}

Our goal is to show a similar fact to Theorem \ref{tw_2/3} for the non-stationary case. However, while computing the velocity matching condition in this case, we arrive at the principle value of an integral that we're unable to compute, see \eqref{calka} below.

\begin{proposition}
Let $\mu =\frac{2}{3}$. The velocity matching condition
\begin{equation}\label{velmatch}
n \cdot (w - \mu z)_{| \Sigma_1} = 0
\end{equation}
holds if and only if
\begin{equation}\label{calka}
\frac{1}{2\pi} PV \int_0^\infty \frac{2\Gamma^3 - 2 \Gamma t^2 \cos \left( \frac{1}{2\pi} \left( t^{-3}-\Gamma^{-3} \right) \right) + \frac{3}{2\pi} \Gamma^{-2}t^2 \sin \left( \frac{1}{2\pi} \left( t^{-3}-\Gamma^{-3} \right) \right)}{\Gamma^4 + t^4 - 2 \Gamma^2 t^2 \cos \left( \frac{1}{2\pi} \left( t^{-3} - \Gamma^{-3} \right) \right)} \, dt = \frac{1}{\pi}
\end{equation}
for any $\Gamma > 0$.
\end{proposition}

\begin{proof}
We recall that $w(z)=v(z,1)$ and in the case $\mu = \frac23$, $\Sigma_1$ is described by
$$
Z(\Gamma, 1) = \Gamma^{2} e^{\frac{i}{2\pi} \Gamma^{-3}}.
$$
Therefore the tangent vector to $\Sigma_1$ at $z = Z(\Gamma,1)$ is
$$
\tau = \frac{d}{d\Gamma} Z(\Gamma, 1) =  2 \Gamma e^{\frac{i}{2\pi} \Gamma^{-3}} - \frac{3i}{2\pi} \Gamma^{-2} e^{\frac{i}{2\pi} \Gamma^{-3}}.
$$
Thus
$$
n = i\tau = i 2 \Gamma e^{\frac{i}{2\pi} \Gamma^{-3}} + \frac{3}{2\pi} \Gamma^{-2} e^{\frac{i}{2\pi} \Gamma^{-3}}.
$$
Observe that
\begin{align*}
n \cdot \mu z  &= \frac{2}{3} \Re (n \overline{z}) = \frac{2}{3} \Re \left( \left( i 2 \Gamma e^{\frac{i}{2\pi} \Gamma^{-3}} + \frac{3}{2\pi} \Gamma^{-2} e^{\frac{i}{2\pi} \Gamma^{-3}} \right) \overline{\Gamma^{2} e^{\frac{i}{2\pi} \Gamma^{-3}}} \right) \\
&= \frac{2}{3} \Re \left( \left( i 2 \Gamma e^{\frac{i}{2\pi} \Gamma^{-3}} + \frac{3}{2\pi} \Gamma^{-2} e^{\frac{i}{2\pi} \Gamma^{-3}} \right) \Gamma^{2} e^{\frac{-i}{2\pi} \Gamma^{-3}} \right) = \frac{2}{3} \Re \left( i 2 \Gamma^3  + \frac{3}{2\pi} \right) = \frac{1}{\pi}.
\end{align*}
Hence \eqref{velmatch} if and only if $n \cdot w_{| \Sigma} = \frac{1}{\pi}$ and it is enough to compute $n \cdot w_{| \Sigma}$. We fix $z = Z(\Gamma,1) \in \Sigma_1$ and we get
\begin{align}
\nonumber n \cdot w &= \Re \left(n \overline{w(z)} \right) = \Re \left(n \overline{v(z,1)} \right) =  \frac{1}{2\pi} \Re \left( n \int_{\C} \frac{ -i (\overline{z}-\overline{y})}{|z-y|^2} \, d\omega_1 (y) \right) \\
\nonumber &= \frac{1}{2\pi} \Re \left( -in PV \int_0^\infty \frac{\overline{z} - s \exp \left( -\frac{i}{2\pi} s^{-3/2} \right)}{\left| z - s \exp \left( \frac{i}{2\pi} s^{-3/2} \right) \right|^2} \cdot \frac12 \cdot s^{-1/2} \, ds \right) \\
\nonumber &= -\frac{1}{4\pi} PV \int_0^\infty \Re \left( in \frac{\overline{z} - s \exp \left( -\frac{i}{2\pi} s^{-3/2} \right)}{\left| z - s \exp \left( \frac{i}{2\pi} s^{-3/2} \right) \right|^2} \right) s^{-1/2} \, ds \\
\nonumber &= - \frac{1}{4\pi} PV \int_0^\infty \frac{\Re \left( in \left( \overline{z} - s \exp \left( -\frac{i}{2\pi} s^{-3/2} \right) \right) \right) }{\left| z - s \exp \left( \frac{i}{2\pi} s^{-3/2} \right) \right|^2} s^{-1/2} \, ds = [t = s^{1/2}, \ 2dt = s^{-1/2} ds] \\
\label{nw} &= - \frac{1}{2\pi} PV \int_0^\infty \frac{\Re \left( in \left( \overline{z} - t^2 \exp \left( -\frac{i}{2\pi} t^{-3} \right) \right) \right) }{\left| z - t^2 \exp \left( \frac{i}{2\pi} t^{-3} \right) \right|^2} \, dt.
\end{align}
Note that
\begin{align}
\nonumber \left| z - t^2 \exp \left( \frac{i}{2\pi} t^{-3} \right) \right|^2 &= |z|^2 + \left| t^2 \exp \left( \frac{i}{2\pi} t^{-3} \right) \right|^2 - 2 \Re \left( \overline{z} t^2 \exp \left( \frac{i}{2\pi} t^{-3} \right) \right) \\
\nonumber &= \Gamma^4 + t^4 - 2 \Re \left( \Gamma^2 t^2 \exp \left( \frac{i}{2\pi} \left( t^{-3} - \Gamma^{-3} \right) \right) \right) \\
\label{norm} &= \Gamma^4 + t^4 - 2 \Gamma^2 t^2 \cos \left( \frac{1}{2\pi} \left( t^{-3} - \Gamma^{-3} \right) \right).
\end{align}
Now we will compute the numerator
\begin{align*}
&\quad \Re \left( in \left( \overline{z} - t^2 \exp \left( -\frac{i}{2\pi} t^{-3} \right) \right) \right) = \Re \left( in \left( \Gamma^{2} e^{-\frac{i}{2\pi} \Gamma^{-3}} - t^2 e^{ -\frac{i}{2\pi} t^{-3} } \right) \right)  \\
&= \Re \left( i \left(i 2 \Gamma e^{\frac{i}{2\pi} \Gamma^{-3}} + \frac{3}{2\pi} \Gamma^{-2} e^{\frac{i}{2\pi} \Gamma^{-3}}\right) \left( \Gamma^{2} e^{-\frac{i}{2\pi} \Gamma^{-3}} - t^2 e^{ -\frac{i}{2\pi} t^{-3} } \right) \right) \\
&= \Re \left( \left(- 2 \Gamma e^{\frac{i}{2\pi} \Gamma^{-3}} + \frac{3i}{2\pi} \Gamma^{-2} e^{\frac{i}{2\pi} \Gamma^{-3}}\right) \left( \Gamma^{2} e^{-\frac{i}{2\pi} \Gamma^{-3}} - t^2 e^{ -\frac{i}{2\pi} t^{-3} } \right) \right) \\
&= \Re \left( \left(- 2 \Gamma + \frac{3i}{2\pi} \Gamma^{-2} \right) \left( \Gamma^{2} - t^2 e^{ -\frac{i}{2\pi} \left( t^{-3} - \Gamma^{-3} \right) } \right) \right) \\
&= \Re \left( -2\Gamma^3 + \frac{3i}{2\pi} + 2 \Gamma t^2 e^{ -\frac{i}{2\pi} \left( t^{-3} - \Gamma^{-3} \right) } - \frac{3i}{2\pi} \Gamma^{-2}t^2 e^{ -\frac{i}{2\pi} \left( t^{-3} - \Gamma^{-3} \right) } \right) \\
&= -2\Gamma^3 + 2 \Gamma t^2 \cos \left( \frac{1}{2\pi} \left( t^{-3}-\Gamma^{-3} \right) \right) - \frac{3}{2\pi} \Gamma^{-2}t^2 \sin \left( \frac{1}{2\pi} \left( t^{-3}-\Gamma^{-3} \right) \right).
\end{align*}
Combining it with \eqref{nw} and \eqref{norm} we arrive at
\begin{align*}
\nonumber n \cdot w &= - \frac{1}{2\pi} PV \int_0^\infty \frac{-2\Gamma^3 + 2 \Gamma t^2 \cos \left( \frac{1}{2\pi} \left( t^{-3}-\Gamma^{-3} \right) \right) - \frac{3}{2\pi} \Gamma^{-2}t^2 \sin \left( \frac{1}{2\pi} \left( t^{-3}-\Gamma^{-3} \right) \right)}{\Gamma^4 + t^4 - 2 \Gamma^2 t^2 \cos \left( \frac{1}{2\pi} \left( t^{-3} - \Gamma^{-3} \right) \right)} \, dt \\
&= \frac{1}{2\pi} PV \int_0^\infty \frac{2\Gamma^3 - 2 \Gamma t^2 \cos \left( \frac{1}{2\pi} \left( t^{-3}-\Gamma^{-3} \right) \right) + \frac{3}{2\pi} \Gamma^{-2}t^2 \sin \left( \frac{1}{2\pi} \left( t^{-3}-\Gamma^{-3} \right) \right)}{\Gamma^4 + t^4 - 2 \Gamma^2 t^2 \cos \left( \frac{1}{2\pi} \left( t^{-3} - \Gamma^{-3} \right) \right)} \, dt.
\end{align*}
\end{proof}

Unfortunatelly, numerical experiments show that the equality \eqref{calka} may not hold. Indeed, consider $\Gamma = 1$ and the following integral 
\begin{align*}
& \frac{1}{2\pi} PV 
\int_0^\infty \frac{2 - 2  t^2
 \cos \big( \frac{1}{2\pi} \big( t^{-3}\!-\!1 \big) \big) 
\!+\! \frac{3}{2\pi}\, t^2
 \sin \big( \frac{1}{2\pi} \big( t^{-3}\!-\!1 \big) \big)}
{1 + t^4 - 2  t^2 
\cos \big( \frac{1}{2\pi} \big( t^{-3}\! -\! 1 \big) \big)} \, dt \\
&\hskip7cm  = \frac{1}{2\pi} \lim_{\varepsilon \to 0^+} \Bigg( \int_0^{1-\varepsilon} + \int_{1+\varepsilon}^{\infty} \Bigg).
\end{align*}
To estimate numerically this integral, we can use \texttt{mpmath} module in Python and compute an approximated value for $\varepsilon \in \{ 10^{-1}, \ldots, 10^{-9} \}$. 

\begin{lstlisting}[caption={Python code for numerical integration}, label={lst:python}]
import mpmath as mp

mp.mp.dps = 100 
pi = mp.pi

def integrand(t):
    theta = (t**(-3) - 1) / (2*pi)
    num = 2 - 2*t**2 * mp.cos(theta) + (3/(2*pi))*t**2 * mp.sin(theta)
    den = 1 + t**4 - 2*t**2 * mp.cos(theta)
    return num/den

def principal_value_integral(eps):
    I1 = mp.quad(integrand, [0, 1 - eps])
    I2 = mp.quad(integrand, [1 + eps, mp.inf])
    return I1 + I2

for eps in [1e-1, 1e-2, 1e-3, 1e-4, 1e-5, 1e-6, 1e-7, 1e-8, 1e-9]:
    I = principal_value_integral(eps)
    result = I/(2*pi)
    print(f"eps = {eps:1.0e}, I/(2pi) = {float(result):.10f}")
\end{lstlisting}

It gives the following result.

\begin{lstlisting}[caption={Result of the code in Listing \ref{lst:python}}, label={lst:result}]
eps = 1e-01, I/(2pi) = -0.0563264347
eps = 1e-02, I/(2pi) = -0.0443320238
eps = 1e-03, I/(2pi) = -0.0431515017
eps = 1e-04, I/(2pi) = -0.0430283775
eps = 1e-05, I/(2pi) = -0.0430172758
eps = 1e-06, I/(2pi) = -0.0430159708
eps = 1e-07, I/(2pi) = -0.0430158447
eps = 1e-08, I/(2pi) = -0.0430158340
eps = 1e-09, I/(2pi) = -0.0430158134
\end{lstlisting}

Although we are unable to present a rigorous proof of this fact, the numerical experiment shows that the value is negative and stabilizes around $-0.0430158$, hence it cannot be equal to $\frac{1}{\pi}$.

\noindent
\textbf{Acknowledgement.} T. Cie\'slak and J. Siemianowski were supported by the SONATA BIS 7 grant UMO-2017/26/E/ST1/00989 from the National Science Centre, Poland.

\end{document}